\documentclass[11pt]{amsart}
\usepackage{latexsym}
\usepackage{fullpage}
\usepackage{amssymb}
\usepackage{amscd}
\usepackage{float}
\usepackage{graphicx}
\usepackage{amsfonts}
\usepackage{pb-diagram}
 \usepackage{amsmath,amscd}

\newtheorem{thm}{Theorem}
\newtheorem{cor}{Corollary}
\newtheorem{lemma}{Lemma}
\newtheorem{prop}{Proposition}
\newtheorem{defn}{Definition}
\newtheorem{remark}{Remark}

\newtheorem{nt}{Notation}

\begin{document}

\title[An alternative basis for the Kauffman bracket skein module of the Solid Torus via braids]
  {An alternative basis for the Kauffman bracket skein module of the Solid Torus via braids}

\author{Ioannis Diamantis}
\address{ International College Beijing,
China Agricultural University,
No.17 Qinghua East Road, Haidian District,
Beijing, {100083}, P. R. China.}
\email{ioannis.diamantis@hotmail.com}



\setcounter{section}{-1}

\date{}

\begin{abstract}
In this paper we give an alternative basis, $\mathcal{B}_{\rm ST}$, for the Kauffman bracket skein module of the solid torus, ${\rm KBSM}\left({\rm ST}\right)$. The basis $\mathcal{B}_{\rm ST}$ is obtained with the use of the Tempereley--Lieb algebra of type B and it is appropriate for computing the Kauffman bracket skein module of the lens spaces $L(p, q)$ via braids.
\end{abstract}

\maketitle

\section{Introduction and overview}\label{intro}

Skein modules were independently introduced by Przytycki \cite{P} and Turaev \cite{Tu} as generalizations of knot polynomials in $S^3$ to knot polynomials in arbitrary 3-manifolds. The essence is that skein modules are quotients of free modules over ambient isotopy classes of links in 3-manifolds by properly chosen local (skein) relations.

\begin{defn}\rm
Let $M$ be an oriented $3$-manifold and $\mathcal{L}_{{\rm fr}}$ be the set of isotopy classes of unoriented framed links in $M$. Let $R=\mathbb{Z}[A^{\pm1}]$ be the Laurent polynomials in $A$ and let $R\mathcal{L}_{{\rm fr}}$ be the free $R$-module generated by $\mathcal{L}_{{\rm fr}}$. Let $\mathcal{S}$ be the ideal generated by the skein expressions $L-AL_{\infty}-A^{-1}L_{0}$ and $L \bigsqcup {\rm O} - (-A^2-A^{-1})L$, where $L_{\infty}$ and $L_{0}$ are represented schematically by the illustrations in Figure~\ref{skein}. Note that blackboard framing is assumed. 

\begin{figure}[!ht]
\begin{center}
\includegraphics[width=1.9in]{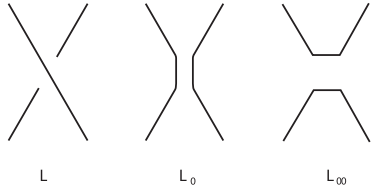}
\end{center}
\caption{The links $L$, $L_{0}$ and $L_{\infty}$ locally.}
\label{skein}
\end{figure}

\noindent Then the {\it Kauffman bracket skein module} of $M$, KBSM$(M)$, is defined to be:

\begin{equation*}
{\rm KBSM} \left(M\right)={\raise0.7ex\hbox{$
R\mathcal{L} $}\!\mathord{\left/ {\vphantom {R\mathcal{L_{{\rm fr}}} {\mathcal{S} }}} \right. \kern-\nulldelimiterspace}\!\lower0.7ex\hbox{$ S  $}}.
\end{equation*}

\end{defn}

 In \cite{Tu} the Kauffman bracket skein module of the solid torus, ST, is computed using diagrammatic methods by means of the following theorem:

\begin{thm}[\cite{Tu}]
The Kauffman bracket skein module of ST, KBSM(ST), is freely generated by an infinite set of generators $\left\{x^n\right\}_{n=0}^{\infty}$, where $x^n$ denotes a parallel copy of $n$ longitudes of ST and $x^0$ is the affine unknot (see Figure~\ref{tur}).
\end{thm}

\begin{figure}
\begin{center}
\includegraphics[width=5in]{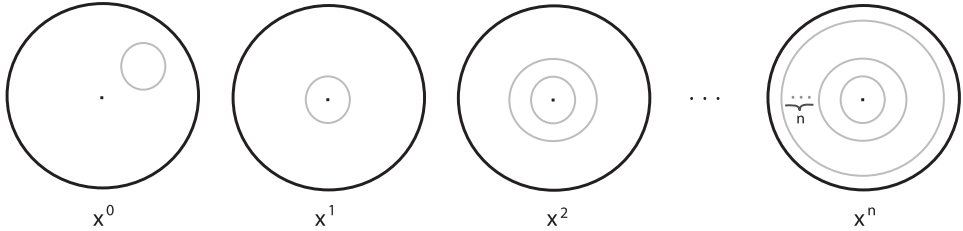}
\end{center}
\caption{The Turaev basis of KBSM(ST).}
\label{tur}
\end{figure}

\smallbreak

In \cite{La2} the most generic analogue of the HOMFLYPT polynomial, $X$, for links in the solid torus $\rm ST$ has been derived from the generalized Hecke algebras of type $\rm B$, $\textrm{H}_{1,n}$, which is related to the knot theory of the solid torus and the Artin group of Coxeter group of type B, $B_{1, n}$, via a unique Markov trace constructed on them. As explained in \cite{La2, DL2}, the Lambropoulou invariant $X$ recovers the HOMFLYPT skein module of ST, $\mathcal{S}({\rm ST})$, and is appropriate for extending the results to the lens spaces $L(p,q)$, since the combinatorial setting is the same as for $\rm ST$, only the braid equivalence includes the braid band moves (shorthanded to bbm), which reflect the surgery description of $L(p,q)$. In \cite{FG} the same procedure is applied for the case of Temperley-Lieb algebras of type B and an invariant $V^{{\rm B}}$ for knots and links in ST is constructed, via a unique Markov trace constructed on them, and which is the analogue of the Kauffman bracket polynomial for knots and links in ST. 

\smallbreak

In this paper the Kauffman bracket skein module of ST, ${\rm KBSM}\left({\rm ST}\right)$, is computed using braids and algebraic techniques developed in \cite{LR1, LR2, La1, La2, DL1, DL2, DLP, DL4, DL5} and \cite{FG}. The motivation of this work is the computation of ${\rm KBSM}\left(L(p, q)\right)$ via algebraic means. The new basic set is described in Eq.~\ref{basis} in terms of mixed braids (that is, classical braids with the first strand identically fixed). For an illustration see bottom part of Figure~\ref{allbases}.

\smallbreak

Our main result is the following:

\begin{thm}\label{newbasis}
The following set forms a basis for KBSM(ST):
\begin{equation}\label{basis}
\mathcal{B}_{\rm ST}\ =\ \{t^{n},\ n \in \mathbb{N} \}.
\end{equation}
\end{thm}

The method for obtaining the basis $\mathcal{B}_{\rm ST}$, is the following:

\begin{itemize}
\item[$\bullet$] We start from elements in the standard basis of KBSM(ST), $\mathcal{B}_{{\rm ST}}^{\prime}$, presented in \cite{Tu}. Then, following the technique in \cite{DL2}, we express these elements into sums of elements in the $\Lambda$, using conjugation and stabilization moves. As shown in \cite{DL2}, the set $\Lambda$ (see Remark~\ref{lam}), forms a basis for the HOMFLYPT skein module of the solid torus. 
\item[$\bullet$] We then express elements in $\Lambda$ to sums of elements in $\mathcal{B}_{{\rm ST}}$, using conjugation, stabilization moves and the Kauffman bracket skein relation.
\item[$\bullet$] We relate the two sets $\mathcal{B}_{{\rm ST}}^{\prime}$ and $\mathcal{B}_{{\rm ST}}$ via an infinite lower triangular matrix and conclude that the set $\mathcal{B}_{{\rm ST}}$ forms a basis for KBSM(ST).
\end{itemize}

The paper is organized as follows: In \S\ref{basics} we recall the setting and the essential techniques and results from \cite{La1, La2, LR1, LR2, DL1}. More precisely, we present isotopy moves for knots and links in ST and we then describe braid equivalence for knots and links in ST. We also present results from \cite{La2} and \cite{FG} and in particular we present the basis of the Kauffman bracket skein module of ST in terms of braids and braid groups of type B. In \S\ref{kbsm} we present results from \cite{DL2} that are crucial for this paper, and using these results, in \S\ref{pr} we present a new basis for the Kauffman bracket skein module of the solid torus ST, $\mathcal{B}_{{\rm ST}}$. As explained in the beginning of \S\ref{kbsm}, the importance of the basis $\mathcal{B}_{{\rm ST}}$ lies in the fact that the {\it braid band moves} or {\it slide moves} (that reflect isotopy in the lens spaces $L(p, q)$) are naturally described via $\mathcal{B}_{{\rm ST}}$. Finally in \cite{D} and starting from $\mathcal{B}_{{\rm ST}}$, the computation of the Kauffman bracket skein module of the lens spaces is presented.

\bigbreak

\noindent \textbf{Acknowledgments}\ \ The author would like to acknowledge several discussions with Professor Sofia Lambropoulou. Moreover, financial support by the China Agricultural University, International College Beijing is gratefully acknowledged.

\section{Preliminaries}\label{basics}

\subsection{Mixed links and isotopy in ST}
 
We consider ST to be the complement of a solid torus in $S^3$. As explained in \cite{LR1, LR2, DL1}, an oriented link $L$ in ST can be represented by an oriented \textit{mixed link} in $S^{3}$, that is a link in $S^{3}$ consisting of the unknotted fixed part $\widehat{I}$ representing the complementary solid torus in $S^3$ and the moving part $L$ that links with $\widehat{I}$. A \textit{mixed link diagram} is a diagram $\widehat{I}\cup \widetilde{L}$ of $\widehat{I}\cup L$ on the plane of $\widehat{I}$, where this plane is equipped with the top-to-bottom direction of $I$ (see right hand side of Figure~\ref{bmov}).

\smallbreak

Consider now an isotopy of an oriented link $L$ in ST. As the link moves in ST, its corresponding mixed link will change in $S^3$ by a sequence of moves that keep the oriented $\widehat{I}$ point-wise fixed. This sequence of moves consists in isotopy in the $S^3$ and the {\it mixed Reidemeister moves}. In terms of diagrams we have the following result for isotopy in ST:

\smallbreak

{\it
The mixed link equivalence in $S^3$ includes the classical Reidemeister moves and the mixed Reidemeister moves, which involve the fixed and the standard part of the mixed link, keeping $\widehat{I}$ pointwise fixed.
}

\subsection{Mixed braids and braid equivalence for knots and links in ST}

By the Alexander theorem for knots and links in the solid torus (cf. Thm.~1 \cite{La2}), a mixed link diagram $\widehat{I}\cup \widetilde{L}$ of $\widehat{I}\cup L$ may be turned into a \textit{mixed braid} $I\cup \beta$ with isotopic closure. This is a braid in $S^{3}$ where, without loss of generality, its first strand represents $\widehat{I}$, the fixed part, and the other strands, $\beta$, represent the moving part $L$. The subbraid $\beta$ is called the \textit{moving part} of $I\cup \beta$ (see left hand side of Figure~\ref{bmov}).

\begin{figure}
\begin{center}
\includegraphics[width=2.3in]{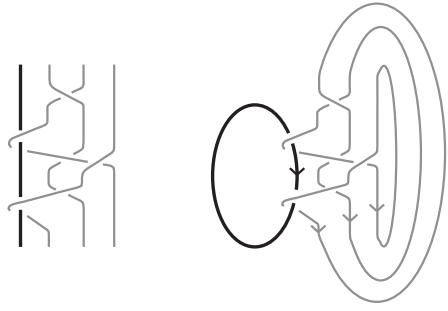}
\end{center}
\caption{The closure of a mixed braid to a mixed link.}
\label{bmov}
\end{figure}

\smallbreak

The sets of braids related to ST form groups, which are in fact the Artin braid groups of type B, denoted $B_{1,n}$, with presentation:

\[ B_{1,n} = \left< \begin{array}{ll}  \begin{array}{l} t, \sigma_{1}, \ldots ,\sigma_{n-1}  \\ \end{array} & \left| \begin{array}{l}
\sigma_{1}t\sigma_{1}t=t\sigma_{1}t\sigma_{1} \ \   \\
 t\sigma_{i}=\sigma_{i}t, \quad{i>1}  \\
{\sigma_i}\sigma_{i+1}{\sigma_i}=\sigma_{i+1}{\sigma_i}\sigma_{i+1}, \quad{ 1 \leq i \leq n-2}   \\
 {\sigma_i}{\sigma_j}={\sigma_j}{\sigma_i}, \quad{|i-j|>1}  \\
\end{array} \right.  \end{array} \right>, \]

\noindent where the generators $\sigma _{i}$ and $t$ are illustrated in Figure~\ref{genh}(i).

%\begin{figure}
%\begin{center}
%\includegraphics[width=2.3in]{gen}
%\end{center}
%\caption{The generators of $B_{1,n}$.}
%\label{gen}
%\end{figure}

Let now $\mathcal{L}$ denote the set of oriented knots and links in ST. Then, isotopy in ST is translated on the level of mixed braids by means of the following theorem:

\begin{thm}[Theorem~5, \cite{LR2}] \label{markov}
 Let $L_{1} ,L_{2}$ be two oriented links in $L(p,1)$ and let $I\cup \beta_{1} ,{\rm \; }I\cup \beta_{2}$ be two corresponding mixed braids in $S^{3}$. Then $L_{1}$ is isotopic to $L_{2}$ in $L(p,1)$ if and only if $I\cup \beta_{1}$ is equivalent to $I\cup \beta_{2}$ in $\mathcal{B}$ by the following moves:
\[ \begin{array}{clll}
(i)  & Conjugation:         & \alpha \sim \beta^{-1} \alpha \beta, & {\rm if}\ \alpha ,\beta \in B_{1,n}. \\
(ii) & Stabilization\ moves: &  \alpha \sim \alpha \sigma_{n}^{\pm 1} \in B_{1,n+1}, & {\rm if}\ \alpha \in B_{1,n}. \\
(iii) & Loop\ conjugation: & \alpha \sim t^{\pm 1} \alpha t^{\mp 1}, & {\rm if}\ \alpha \in B_{1,n}. \\
\end{array} \]

\end{thm}

\subsection{The Kauffman bracket skein module of ST via braids}\label{SolidTorus}

In \cite{La2} the most generic analogue of the HOMFLYPT polynomial, $X$, for links in the solid torus $\rm ST$ has been derived from the generalized Iwahori--Hecke algebras of type $\rm B$, $\textrm{H}_{1,n}$, via a unique Markov trace constructed on them. This algebra was defined by Lambropoulou as the quotient of ${\mathbb C}\left[q^{\pm 1} \right]B_{1,n}$ over the quadratic relations ${g_{i}^2=(q-1)g_{i}+q}$. Namely:

\begin{equation*}
\textrm{H}_{1,n}(q)= \frac{{\mathbb C}\left[q^{\pm 1} \right]B_{1,n}}{ \langle \sigma_i^2 -\left(q-1\right)\sigma_i-q \rangle}.
\end{equation*}

It is also shown that the following sets form linear bases for ${\rm H}_{1,n}(q)$ (\cite[Proposition~1 \& Theorem~1]{La2}):

\begin{equation}
\begin{array}{llll}
 (i) & \Sigma_{n} & = & \{t_{i_{1} } ^{k_{1} } \ldots t_{i_{r}}^{k_{r} } \cdot \sigma \} ,\ {\rm where}\ 0\le i_{1} <\ldots <i_{r} \le n-1,\\
 (ii) & \Sigma^{\prime} _{n} & = & \{ {t^{\prime}_{i_1}}^{k_{1}} \ldots {t^{\prime}_{i_r}}^{k_{r}} \cdot \sigma \} ,\ {\rm where}\ 0\le i_{1} < \ldots <i_{r} \le n-1, \\
\end{array}
\end{equation}
\noindent where $k_{1}, \ldots ,k_{r} \in {\mathbb Z}$, $t_0^{\prime}\ =\ t_0\ :=\ t, \quad t_i^{\prime}\ =\ g_i\ldots g_1tg_1^{-1}\ldots g_i^{-1} \quad {\rm and}\quad t_i\ =\ g_i\ldots g_1tg_1\ldots g_i$ are the `looping elements' in ${\rm H}_{1, n}(q)$ (see Figure~\ref{genh}(ii)) and $\sigma$ a basic element in the Iwahori--Hecke algebra of type A, ${\rm H}_{n}(q)$, for example in the form of the elements in the set \cite{Jo}:

$$ S_n =\left\{(g_{i_{1} }g_{i_{1}-1}\ldots g_{i_{1}-k_{1}})(g_{i_{2} }g_{i_{2}-1 }\ldots g_{i_{2}-k_{2}})\ldots (g_{i_{p} }g_{i_{p}-1 }\ldots g_{i_{p}-k_{p}})\right\}, $$

\noindent for $1\le i_{1}<\ldots <i_{p} \le n-1{\rm \; }$. In \cite{La2} the bases $\Sigma^{\prime}_{n}$ are used for constructing a Markov trace on $\mathcal{H}:=\bigcup_{n=1}^{\infty}{\rm H}_{1, n}$, and using this trace, a universal HOMFLYPT-type invariant for oriented links in ST is constructed.

\begin{figure}\label{loopttpr}
\begin{center}
\includegraphics[width=5.5in]{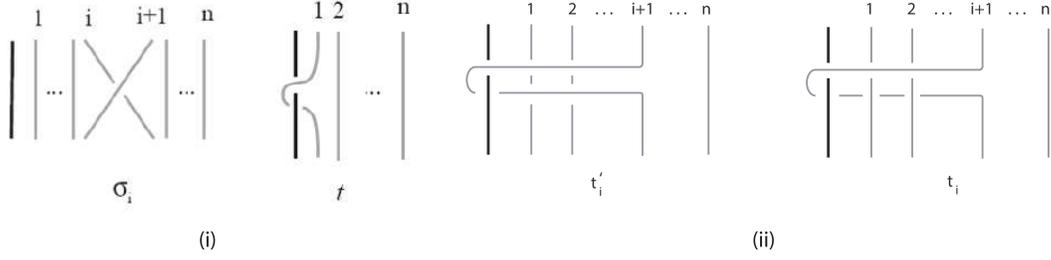}
\end{center}
\caption{The generators of $B_{1, n}$ and the `looping' elements $t^{\prime}_{i}$ and $t_{i}$.}
\label{genh}
\end{figure}

\begin{thm}{\cite[Theorem~6 \& Definition~1]{La2}} \label{tr}
Given $z, s_{k}$ with $k\in {\mathbb Z}$ specified elements in $R={\mathbb C}\left[q^{\pm 1} \right]$, there exists a unique linear Markov trace function on $\mathcal{H}$:

\begin{equation*}
{\rm tr}:\mathcal{H}  \to R\left(z,s_{k} \right),\ k\in {\mathbb Z}
\end{equation*}

\noindent determined by the rules:

\[
\begin{array}{lllll}
(1) & {\rm tr}(ab) & = & {\rm tr}(ba) & \quad {\rm for}\ a,b \in {\rm H}_{1,n}(q) \\
(2) & {\rm tr}(1) & = & 1 & \quad {\rm for\ all}\ {\rm H}_{1,n}(q) \\
(3) & {\rm tr}(ag_{n}) & = & z{\rm tr}(a) & \quad {\rm for}\ a \in {\rm H}_{1,n}(q) \\
(4) & {\rm tr}(a{t^{\prime}_{n}}^{k}) & = & s_{k}{\rm tr}(a) & \quad {\rm for}\ a \in {\rm H}_{1,n}(q),\ k \in {\mathbb Z} \\
\end{array}
\]

\bigbreak

\noindent Then, the function $X:\mathcal{L}$ $\rightarrow R(z,s_{k})$

\begin{equation*}
X_{\widehat{\alpha}} = \Delta^{n-1}\cdot \left(\sqrt{\lambda } \right)^{e}
{\rm tr}\left(\pi \left(\alpha \right)\right),
\end{equation*}

\noindent is an invariant of oriented links in {\rm ST}, where $\Delta:=-\frac{1-\lambda q}{\sqrt{\lambda } \left(1-q\right)}$, $\lambda := \frac{z+1-q}{qz}$, $\alpha \in B_{1,n}$ is a word in the $\sigma _{i}$'s and $t^{\prime}_{i} $'s, $\widehat{\alpha}$ is the closure of $\alpha$, $e$ is the exponent sum of the $\sigma _{i}$'s in $\alpha $, $\pi$ the canonical map of $B_{1,n}$ on ${\rm H}_{1,n}(q)$, such that $t\mapsto t$ and $\sigma _{i} \mapsto g_{i}$.
\end{thm}

\begin{remark}\rm
As shown in \cite{La2, DL2} the invariant $X$ recovers the HOMFLYPT skein module of ST. For a survey on the HOMFLYPT skein module of the lens spaces $L(p, 1)$ via braids, the reader is referred to \cite{DL3}.
\end{remark}

\bigbreak

Following the same idea as in \cite{La2}, in \cite{FG} the analogue of the Kauffman bracket polynomial, $V$, for links in the solid torus $\rm ST$ has been derived from the Temperley-Lieb algebra of type B, $\textrm{TL}_{n}^{{\rm B}}$. This algebra is defined as a quotient of the generalized Iwahori-Hecke algebra of type B, $\textrm{H}_{1,n}(q)$, over the ideal generated by the elements:

\begin{equation}\label{ideal}
\begin{array}{lcl}
h_{1, 2} & := & 1 + u\ (\sigma_1+ \sigma_2) + u^2\ (\sigma_1\sigma_2+\sigma_2\sigma_1)+u^3\ \sigma_1\sigma_2\sigma_1,\qquad {\rm for\ all}\ 1\leq i \leq n-2\\
h_B & := & 1 + u\ \sigma_1 + v\ t + uv\ (\sigma t + t \sigma) + u^2v\ \sigma t \sigma + uv^2\ t \sigma t + (uv)^2\ \sigma t \sigma t
\end{array}
\end{equation}

Note that in \cite{FG} a different presentation for ${\rm H}_{1, n}$ is used, that involves the parameters $u, v$ and the quadratic equations 

\begin{equation}\label{quad}
{\sigma_{i}^2=(u-u^{-1})\sigma_{i}+1}.
\end{equation}

\noindent One can switch from one presentation to the other by a taking $\sigma_i = u \sigma_i$, $t = v t$ and $q = u^2$.

\smallbreak

Since the Temperley-Lieb algebra of type B is a quotient of the Iwahori-Hecke algebra of type B, in \cite{FG} the authors present the necessary and sufficient conditions so as the Markov trace factors through to ${\rm TL}_{n}^{{\rm B}}$. Indeed:

\begin{thm}{\cite[Theorem~4]{FG}} \label{tr2}
The trace defined in ${\rm H}_n(1, q)$ factors through to ${\rm TL}_{n}^{{\rm B}}$ if and only if the trace parameters take the following values:
\begin{equation}\label{parameters}
z=-\frac{1}{u(1+u^2)}, \qquad s_1=\frac{-1+v^2}{(1+u^2)v}.
\end{equation}
\end{thm} 

It is worth mentioning that in \cite{FG} more values of the trace parameters that allow the trace to factor through to ${\rm TL}_{n}^{{\rm B}}$ are presented, but as explained in \cite{FG}, only the values in (\ref{parameters}) are of topological interest. Moreover, for those values of the parameters one deduces $\lambda\ =\ u^4$. We have the following:

\begin{thm}{\cite{FG}}
The following is an invariant for knots and links in {\rm ST}:

\begin{equation*}
V^{{\rm B}}_{\widehat{\alpha}}(u ,v)\ :=\ \left(-\frac{1+u^2}{u} \right)^{n-1} \left(u\right)^{2e} {\rm tr}\left(\overline{\pi} \left(\alpha \right)\right),
\end{equation*}

\noindent where $\alpha \in B_{1,n}$ is a word in the $\sigma _{i}$'s and $t^{\prime}_{i} $'s, $\widehat{\alpha}$ is the closure of $\alpha$, $e$ is the exponent sum of the $\sigma _{i}$'s in $\alpha $, $\overline{\pi}$ the canonical map of $B_{1,n}$ on ${\rm TL}_{n}^{{\rm B}}$, such that $t\mapsto t$ and $\sigma _{i} \mapsto g_{i}$.
\end{thm}

\bigbreak

In the braid setting of \cite{La2}, the elements of KBSM(ST) correspond bijectively to the elements of the following set:

\begin{equation}\label{Lpr}
\mathcal{B}^{\prime}_{{\rm ST}}=\{ t{t^{\prime}_1} \ldots {t^{\prime}_n}, \ n\in \mathbb{N} \}.
\end{equation}

\noindent The set $\mathcal{B}^{\prime}_{{\rm ST}}$ forms a basis of KBSM(ST) in terms of braids (see also \cite{Tu}). Note that $\mathcal{B}^{\prime}_{{\rm ST}}$ is a subset of $\mathcal{H}$ and, in particular, $\mathcal{B}^{\prime}_{{\rm ST}}$ is a subset of $\Sigma^{\prime}=\bigcup_n\Sigma^{\prime}_n$. Note also that in contrast to elements in $\Sigma^{\prime}$, the elements in $\mathcal{B}^{\prime}_{{\rm ST}}$ have no gaps in the indices, the exponents are all equal to one and there are no `braiding tails'. 

\begin{remark}\rm
The invariant $V^{{\rm B}}$ recovers KBSM(ST). Indeed, it gives distinct values to distinct elements of $\mathcal{B}^{\prime}_{{\rm ST}}$, since ${\rm tr}(t{t^{\prime}_1} \ldots {t^{\prime}_n})=s_{1}^n$.
\end{remark}

\section{The basis $\mathcal{B}_{{\rm ST}}$ of KBSM(ST)}\label{kbsm}

In this section we prove the main result of this paper, Theorem~\ref{newbasis}. Before proceeding with the proof we present the motivation that lead to the new basis $\mathcal{B}_{\rm ST}$ of KBSM(ST):

\smallbreak

The relation between ${\rm KBSM}\left(L(p, 1)\right)$ and ${\rm KBSM}({\rm ST})$ is presented in \cite{P} and it is shown that: 

$${\rm KBSM}\left(L(p, 1)\right)=\frac{{\rm KBSM}({\rm ST})}{<a-bbm(a)>}, \quad a\ {\rm in\ the\ basis\ of\ KBSM(ST)}.$$

In order to extend $V^{{\rm B}}$ to an invariant of links in $L(p, q)$ we need to solve an infinite system of equations resulting from the braid band moves. Namely we force:

\begin{equation}\label{eqbbm}
V^{{\rm B}}_{\widehat{\alpha}}\ =\  V^{{\rm B}}_{\widehat{bbm(\alpha)}},
\end{equation}

\noindent for all $\alpha$ in the basis of KBSM(ST).

\bigbreak

The above equations have particularly simple formulations with the use of the new basis, $\mathcal{B}_{{\rm ST}}$, for the Kauffman bracket skein module of ST. This is a very technical and difficult task and is the subject of a sequel paper.

\smallbreak

We now recall results from \cite{DL2} that we will use throughout the paper. 

\begin{figure}
\begin{center}
\includegraphics[width=4.8in]{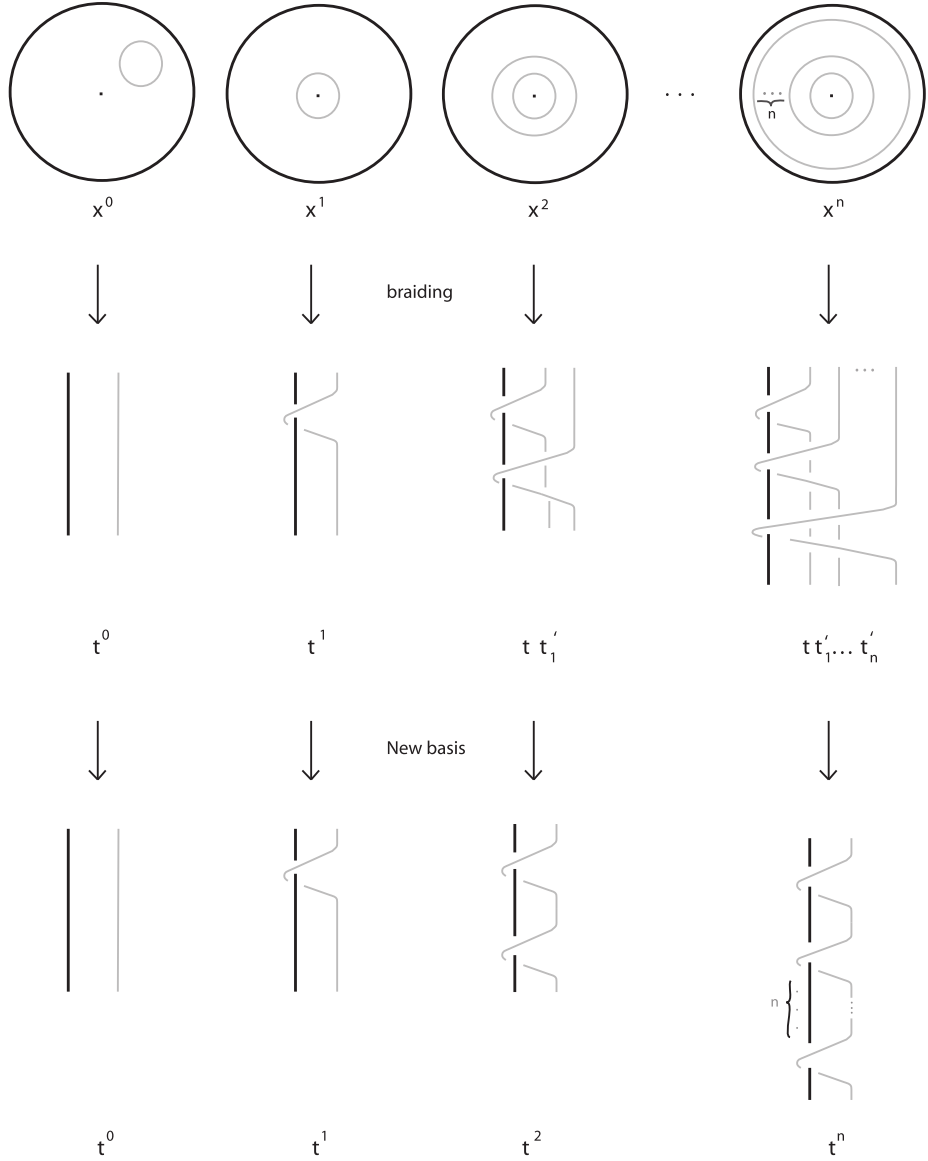}
\end{center}
\caption{Elements in the different bases of KBSM(ST).}
\label{allbases}
\end{figure}

\subsection{An ordering in the bases of $\mathcal{S}({\rm ST})$}

In \cite{DL2} an ordering relation is defined on the sets $\Sigma$ and $\Sigma^{\prime}$ which plays a crucial role in this paper. Before presenting this ordering relation, we first introduce the sets $\Lambda^{\prime}$ and $\Lambda$ and the notion of the {\it index} of a word $w$, denoted $ind(w)$, in any of these sets. 

\begin{defn} \rm
We define the following subsets of $\Sigma_n$ and $\Sigma^{\prime}$ respectively:
\begin{equation}
\begin{array}{l}
\Lambda_{(k)}\ :=\{t_0^{k_0}t_1^{k_1}\ldots t_{m}^{k_m} \ | \ k_i\ \geq\ k_{i+1},\ \sum_{i=0}^{m}{k_i}=k,\ k_i \in \mathbb{Z}\setminus\{0\},\ \forall i \},\\
\\
\Lambda^{\prime}_{(k)}:=\{{t^{\prime}_0}^{k_0}{t^{\prime}_1}^{k_1}\ldots {t^{\prime}_{m}}^{k_m}\ | \ k_i\ \geq\ k_{i+1}, \ \sum_{i=0}^{m}{k_i}=k,\ k_i \in \mathbb{Z}\setminus\{0\},\ \forall i \},\\
\\
\Lambda^{aug}_{(k)}\ :=\{t_0^{k_0}t_1^{k_1}\ldots t_{m}^{k_m} \ | \ \sum_{i=0}^{m}{k_i}=k,\ k_i \in \mathbb{Z}\setminus\{0\},\ \forall i \},\\
\\
{\Lambda^{\prime}_{(k)}}^{aug}:=\{{t^{\prime}_0}^{k_0}{t^{\prime}_1}^{k_1}\ldots {t^{\prime}_{m}}^{k_m}\ | \ \sum_{i=0}^{m}{k_i}=k,\ k_i \in \mathbb{Z}\setminus\{0\},\ \forall i \}.
\end{array}
\end{equation}
\end{defn}

Note that elements in the set $\Lambda_{(k)}$ have ordered exponents, while elements in $\Lambda^{aug}_{(k)}$ have arbitrary exponents. Obviously, $\Lambda_{(k)}\subset \Lambda^{aug}_{(k)}\subset \Sigma_n$.

\begin{remark}\label{lam}\rm
In \cite{DL2} the set $\Lambda\ :=\ \underset{k}{\bigcup}\ \Lambda_{(k)}$ is showed to be a basis for the HOMFLYPT skein module of ST.
\end{remark}

\begin{defn}{\cite[Definition~1]{DL2}} \rm 
Let $w$ a word in $\Lambda$. Then, the index of $w$, $ind(w)$, is defined to be the highest index of the $t_i$'s in $w$. Similarly, in $\Sigma^{\prime}$ or $\Sigma$, $ind(w)$ is defined as above by ignoring possible gaps in the indices of the looping generators and by ignoring the braiding parts in the algebras $\textrm{H}_{n}(q)$. Moreover, the index of a monomial in $\textrm{H}_{n}(q)$ is equal to $0$.
\end{defn}

We now proceed with presenting an ordering relation in the sets $\Sigma$ and $\Sigma^{\prime}$, which passes to their respective subsets $\mathcal{B}_{\rm ST}$ and $\mathcal{B}_{\rm ST}^{\prime}$.

\begin{defn}{\cite[Definition~2]{DL2}} \label{order}\rm
Let $w={t^{\prime}_{i_1}}^{k_1}\ldots {t^{\prime}_{i_{\mu}}}^{k_{\mu}}\cdot \beta_1$ and $u={t^{\prime}_{j_1}}^{\lambda_1}\ldots {t^{\prime}_{j_{\nu}}}^{\lambda_{\nu}}\cdot \beta_2$ in $\Sigma^{\prime}$, where $k_t , \lambda_s \in \mathbb{Z}$ for all $t,s$ and $\beta_1, \beta_2 \in H_n(q)$. Then, we define the following ordering in $\Sigma^{\prime}$:

\smallbreak

\begin{itemize}
\item[(a)] If $\sum_{i=0}^{\mu}k_i < \sum_{i=0}^{\nu}\lambda_i$, then $w<u$.

\vspace{.1in}

\item[(b)] If $\sum_{i=0}^{\mu}k_i = \sum_{i=0}^{\nu}\lambda_i$, then:

\vspace{.1in}

\noindent  (i) if $ind(w)<ind(u)$, then $w<u$,

\vspace{.1in}

\noindent  (ii) if $ind(w)=ind(u)$, then:

\vspace{.1in}

\noindent \ \ \ \ ($\alpha$) if $i_1=j_1, \ldots , i_{s-1}=j_{s-1}, i_{s}<j_{s}$, then $w>u$,

\vspace{.1in}

\noindent \ \ \  ($\beta$) if $i_t=j_t$ for all $t$ and $k_{\mu}=\lambda_{\mu}, k_{\mu-1}=\lambda_{\mu-1}, \ldots, k_{i+1}=\lambda_{i+1}, |k_i|<|\lambda_i|$, then $w<u$,

\vspace{.1in}

\noindent \ \ \  ($\gamma$) if $i_t=j_t$ for all $t$ and $k_{\mu}=\lambda_{\mu}, k_{\mu-1}=\lambda_{\mu-1}, \ldots, k_{i+1}=\lambda_{i+1}, |k_i|=|\lambda_i|$ and $k_i>\lambda_i$, then $w<u$,

\vspace{.1in}

\noindent \ \ \ \ ($\delta$) if $i_t=j_t\ \forall t$ and $k_i=\lambda_i$, $\forall i$, then $w=u$.

\end{itemize}

The ordering in the set $\Sigma$ is defined as in $\Sigma^{\prime}$, where $t_i^{\prime}$'s are replaced by $t_i$'s.
\end{defn}

%\begin{nt}\label{nt} \rm
%We set $\tau_{i,i+m}^{k_{i,i+m}}:=t_i^{k_i}\ldots t^{k_{i+m}}_{i+m}$ and ${\tau^{\prime}}_{i,i+m}^{k_{i,i+m}}:={t^{\prime}}_i^{k_i}\ldots %{t^{\prime}}^{k_{i+m}}_{i+m}$, for $m\in \mathbb{N}$, $k_j\neq 0$ for all $j$.
%\end{nt}

\subsection{From $\mathcal{B}^{\prime}_{\rm ST}$ to $\Lambda$}

In this subsection we recall a series of results from \cite{DL2} in order to convert elements in $\mathcal{B}^{\prime}_{\rm ST}$ to elements in $\Lambda$. In order to simplify the algebraic expressions obtained throughout this procedure and throughout the paper in general, we first introduce the following notation:

\begin{nt}\label{nt} \rm
We set $\tau_{i,i+m}^{k_{i,i+m}}:=t_i^{k_i}\ldots t^{k_{i+m}}_{i+m}\in \Sigma$ and ${\tau^{\prime}}_{i,i+m}^{k_{i,i+m}}:={t^{\prime}}_i^{k_i}\ldots {t^{\prime}}^{k_{i+m}}_{i+m}\in \Sigma_n^{\prime}$, for $m\in \mathbb{N}$, $k_j\neq 0$ for all $j$.
\end{nt}

\begin{remark}\rm
Using Notation~\ref{nt}, elements in $\mathcal{B}_{\rm ST}^{\prime}$ are of the form $\tau^{\prime}_{0, n}\ :=\ tt_1^{\prime}\ldots t_n^{\prime}$, for $n\in \mathbb{N}$, that is $\mathcal{B}_{\rm ST}^{\prime}\ =\ \left\{ \tau^{\prime}_{0, n} \right\}_{n=0}^{\infty}$. Moreover, we set $\mathcal{K}_{\rm ST} \ =\ \left\{ \tau_{0, n} \right\}_{n=0}^{\infty}$, and so elements in $\mathcal{K}_{\rm ST}$ are of the form $\tau_{0, n}\ :=\ tt_1\ldots t_n$, for $n\in \mathbb{N}$.

Moreover,
\[
\begin{array}{lcll}
\Lambda^{\prime}_{(k)} & = & \left\{ {\tau^{\prime}}^{k_{0, n}}_{0, n}\ |\ k_i\geq k_{i-1},\ \underset{i=0}{\overset{n}{\sum}}k_i=k, \ k_i\in \mathbb{Z}\backslash \{0\} \right\}, & \Lambda^{\prime}\ =\ \underset{k\in \mathbb{Z}}{\oplus} \Lambda^{\prime}_{(k)}\\
&&&\\
\Lambda_{(k)} & = & \left\{ {\tau}^{k_{0, n}}_{0, n}\ |\ k_i\geq k_{i-1},\ \underset{i=0}{\overset{n}{\sum}}k_i=k, \ k_i\in \mathbb{Z}\backslash \{0\} \right\}, & \Lambda\ =\ \underset{k\in \mathbb{Z}}{\oplus} \Lambda_{(k)}
\end{array}
\]
\end{remark}

\smallbreak

We also introduce the notion of {\it homologous words}, which is crucial for relating the sets $\mathcal{B}_{\rm ST}^{\prime}$ and $\mathcal{K}_{\rm ST}$ via a triangular matrix.

\begin{defn}{\cite[Definition~4]{DL2}} \rm
		We say that two words $w^{\prime}\in \Sigma^{\prime}$ and $w\in \Sigma$ are {\it homologous}, denoted $w^{\prime}\sim w$, if $w$ is
		obtained from $w^{\prime}$ by changing $t^{\prime}_i$ into $t_i$ for all $i$ and ignoring the braiding parts.
	\end{defn}

We now recall a result from \cite{DL2} in order to convert monomials in the $t_i^{\prime}$'s in general to monomials in the $t_i$'s in $\Sigma_n$. More precisely:

\begin{thm}{\cite[Theorem~7]{DL2}}\label{convert}
The following relations hold:

\begin{equation*}\label{conv}
{\tau^{\prime}}^{k_{0, n}}_{0, n} \ = \ \tau^{k_{0, n}}_{0, n}\ + \ A\cdot \tau_{0, n}\cdot w \ + \ \sum_{j}{B_j\tau_j \cdot \beta_j},\\
\end{equation*}

\noindent where $w, \beta_j \in {\rm H}_{n+1}(q), \forall j$, $\tau_j \in \Sigma_n$, such that $\tau_j < \tau_{0, n}, \forall j$ and $A, B_j$ coefficients.
\end{thm}

Since now we are only interested in converting elements in the set $\mathcal{B}_{\rm ST}^{\prime}$ to sums of monomials in the $t_i$'s, we have the following corollary:

\begin{cor}\label{convert1}
The following relations hold:

\begin{equation}\label{conv1}
\tau_{0, n}^{\prime} \ = \ \tau_{0, n}\ + \ A\cdot \tau_{0, n}\cdot w \ + \ \sum_{j}{B_j\tau_j \cdot \beta_j},\\
\end{equation}

\noindent where $w, \beta_j \in {\rm H}_{n+1}(q), \forall j$, $\tau_j \in \Sigma_n$, such that $\tau_j < \tau_{0, n}, \forall j$ and $A, B_j$ coefficients.
\end{cor}

After expressing an element $\tau_{0, n}^{\prime} \in \mathcal{B}_{\rm ST}^{\prime}$ as sums of elements in $\Sigma_n$, we obtain the homologous word $\tau_{0, n}$, the homologous word again followed by a `braiding tail' $w \in {\rm TL}_{n}$ and elements in $\Sigma_n$ with possible `gaps' in the indices. In \cite{DL2}, using conjugation, monomials in the $t_i$'s with `gaps' in the indices are expressed as sums of monomials in $\Lambda$, followed by `braiding tails'. For the expressions that we obtain after appropriate conjugations we shall use the notation $ \widehat{=}$. We recall the following result from \cite{DL2}:

\begin{thm}{\cite[Theorem~8]{DL2}}\label{gap}
Let $T$ be a monomial in the $t_i$'s with `gaps' in the indices. The following relations hold:

\begin{equation}\label{gaps}
T \ \widehat{=} \ \sum_{i}{A_i\cdot T_i\cdot w_i},
\end{equation}

\noindent where $T_i \in \Lambda_{(n)}$, such that $T_i < T, \forall i$, $w_i \in {\rm TL}_{n+1}, \forall i$,  and $A_i$ coefficients.
\end{thm}

As shown in \cite{DL2}, elements in the set $\Lambda$ followed by `braiding tails' can be expressed as sums of elements in $\Lambda^{aug}$ by using conjugation and stabilization moves. For the expressions that we obtain after appropriate conjugations and stabilization moves we shall use the notation $ \widehat{\simeq}$. Indeed, we have the following:

\begin{thm}{\cite[Theorem~10]{DL2}} \label{tail}
Let $\tau \in \Lambda$ and $w\in {\rm TL}_{n}$. Then, applying conjugation and stabilization moves we have that:

\begin{equation}\label{tails}
\tau\cdot w\ \widehat{\simeq}\  \sum_{j}{A_j\cdot \tau_{j}},
\end{equation}

\noindent where $\Lambda_{(n)}\ \ni \ \tau_j < \tau$, for all $j$.
\end{thm}

Combining now Theorems~\ref{convert}, \ref{gap} \& \ref{tail} and Corollary~\ref{convert1} we have that an element $\tau^{\prime} \in \mathcal{B}_{\rm ST}^{\prime}$ can be expressed as a sum of the homologous word $\tau \in \mathcal{K}_{{\rm ST}}$ and lower order terms in $\Lambda_{(n)}$. More precisely, we have the following:

\begin{cor}\label{b'tol}
Let ${\tau^{\prime}}_{0, n} \in \mathcal{B}_{\rm ST}^{\prime}$. The following relations hold:
\begin{equation}
{\tau^{\prime}}_{0, n}\ \ \widehat{\simeq}\ \ {\tau}_{0, n}\ +\ \underset{i}{\sum} A_i\cdot \tau_i,
\end{equation}
\noindent where $\tau_i \in \Lambda$ such that $\tau_i < \tau_{0, n} \sim {\tau^{\prime}}_{0, n}$ for all $i$. 
\end{cor}

From Corollary~\ref{b'tol} we have that monomials ${\tau^{\prime}}_{0, n} \in \mathcal{B}_{\rm ST}$ can be expressed as sums of their corresponding homologous word $\tau_{0, n} \in \mathcal{K}_{\rm ST}$ with invertible coefficients, and elements $\tau_i \in \Lambda$ of lower order than $\tau_{0, n}$. The point now is that the elements $\tau_i$ do not necessarily belong to $\mathcal{K}_{\rm ST}$, but using conjugation and stabilization moves, we will show that these elements can be expressed as monomials in $\mathcal{B}_{\rm ST}$ of lower order than $\tau_{0, n}$, and thus, $\mathcal{B}_{\rm ST}$ spans KBSM(ST). We deal with these elements in the next subsection.

\subsection{From $\Lambda$ to $\mathcal{B}_{\rm ST}$}

As explained in the Introduction, our goal is to relate the sets $\mathcal{B}_{\rm ST}^{\prime}$ and $\mathcal{B}_{\rm ST}$ via an infinite block diagonal, invertible matrix. From Corollary~\ref{b'tol} we have that an element in $\mathcal{B}_{\rm ST}^{\prime}$ can be expressed as a sum of the homologous word in $\mathcal{K}_{\rm ST} \subset \Lambda$ and elements in $\Lambda$ of lower order. In this subsection we convert elements in $\Lambda$ to sums of elements in $\mathcal{B}_{\rm ST}$. We first deal with the homologous word $\tau_{0, n}\in \Lambda$ of ${\tau^{\prime}}_{0, n} \in \mathcal{B}^{\prime}_{{\rm ST}}$. We have the following:

\begin{prop}\label{ltob1}
Applying conjugation, stabilization moves and relations \ref{ideal}, the following relations hold:
\begin{equation}
\Lambda \ni \tau_{0, n}\ \widehat{\simeq}\ A\cdot t^{ind(\tau_{0, n})+1}\ +\ \underset{i=0}{\overset{ind(\tau_{0, n})}{\sum}} A_i\cdot t^i,
\end{equation}
\noindent where $A_i$ coefficients in the ground ring for all $i$.
\end{prop}

\begin{proof}
We prove Proposition~\ref{ltob1} by strong induction on the order of $\tau_{0, n}$.

\smallbreak

The base of the induction is the monomial $tt_1\in \Lambda$ of index $1$. We have that:
\[
\begin{array}{lcl}
tt_1 & = & t \sigma_1 t \sigma_1\ =\ \sigma_1 t \sigma_1 t\ =\\
&&\\
& = & -\frac{1}{(uv)^2} (1 + u \sigma_1 + v t + uv (\sigma t + t \sigma) + u^2v \sigma t \sigma + uv^2 t \sigma t)\ \widehat{\simeq}\\
&&\\
 & \widehat{\simeq} & -\frac{1}{(uv)^2} (1 + uz + v\ t + 2uvz t + u^2v\ t \sigma_1^2 + uv^2\ t^2 \sigma_1)\ \widehat{\simeq}\\
&&\\
 & \widehat{\simeq} & -\frac{1}{(uv)^2} (1 + uz + v\ t + 2uvz t + u^2v\ t  + u^2vz(u-u^{-1})\ t + uv^2z\ t^2)\ =\\
&&\\
& = & (-u^{-1}z)\ t^2\ +\ \underset{i=0}{\overset{1}{\sum}} A_i\cdot t^i.
\end{array}
\]
\noindent So, Proposition~\ref{ltob1} holds for $tt_1$.

\smallbreak

Assume now that Proposition~\ref{ltob1} holds for all monomials $\tau_i$ of lower order than $\tau_{0, n}$. Then, we have:
\[
\begin{array}{lcl}
\tau_{0, n} & := & tt_1\ (\tau_{2, n})\ =\ (t \sigma_1 t \sigma_1)\ (\tau_{2, n})\ =\ (\sigma_1 t \sigma_1 t)\ (\tau_{2, n})\ =\\
&&\\
& = &-\frac{1}{(uv)^2} \left[\ 1 + u \sigma_1 + v t + uv (\sigma_1 t + t \sigma_1) + u^2v \sigma_1 t \sigma_1 + uv^2 t \sigma t\ \right]\ (\tau_{2, n})\ \widehat{\simeq}\\
&&\\
 & \widehat{\simeq} & -\frac{1}{(uv)^2} \left[\ \tau_{2, n} + u\tau_{2, n}\sigma_1 + v t\tau_{2, n} + 2uvt\tau_{2, n}\sigma_1 + u^2v t\tau_{2, n} \sigma_1^2 + uv^2t^2 \tau_{2, n} \sigma_1 \ \right]\ \widehat{\simeq}\\
&&\\
 & \widehat{\simeq} & -\frac{1}{(uv)^2} t^2 \tau_{2, n}\sigma_1\ +\ \underset{i}{\sum} A_i\cdot \tau_i,\ {\rm where}\ \tau_i<\tau,\ \forall\ i.
\end{array}
\]

According to the ordering relation, on the right hand side of this equation we have the element $t^2 \tau_{2, n}\sigma_1$ and a sum of elements of lower order than $\tau_{0, n}$, since the sums of the exponents in the $t_i$'s in these elements are less than $n+1$. Moreover, the monomial $t^2 \tau_{2, n}\sigma_1$ contains a gap in the indices, and thus it is of lower order than $\tau_{0, n}$ (recall Definition~\ref{order}). Moroever, this monomial is followed by the `braiding tail' $\sigma_1$. According now to Theorems~\ref{gap} \& \ref{tail}, this element can be expressed as a sum of elements in $\Lambda_{(n)}$ of lower order than $t^2 \tau_{2, n}\sigma_1$ and hence, of lower order than $\tau_{0, n}$. By the induction hypothesis the proof is now concluded.
\end{proof}

We now deal with arbitrary elements in $\Lambda$ and convert them to sums of elements in $\mathcal{B}({\rm ST})$. We will need the following lemmas:

\begin{lemma}\label{lem1}
The following relations hold for all $n\in \mathbb{N}$:

$$
t^n t_1 \ \widehat{\simeq} \ -\frac{1}{u} z\ t^{n+1}\ +\ \underset{i=n-1}{\overset{n}{\sum}} A_i t^i, 
$$
\noindent where $A_i$ coefficients for all $i$.

\end{lemma}

\begin{proof}
We prove Lemma~\ref{lem1} by induction on $n$. For $n=1$ we have: $tt_1=-\frac{1}{u} z\ t^{n+1}\ +\ \underset{i=0}{\overset{1}{\sum}} A_i t^i$ (relations \ref{ideal}). Assume now that the relation is true for $n$. Then for $n+1$ we have:

\[
t^{n+1}t_1 \ = \ t\cdot (t^n t_1) \ \underset{hyp.}{\overset{ind.}{\widehat{\simeq}}} \ -\frac{1}{u} z\ t^{n+2}\ +\ \underset{i=n-1}{\overset{n}{\sum}} A_i t^{i+1}\ = \ -\frac{1}{u} z\ t^{n+2}\ +\ \underset{i=n}{\overset{n+1}{\sum}} A_i t^i.
\]

\end{proof}

The following lemma will serve as a basis for the induction hypothesis applied in the proof of the main result of this section.

\begin{lemma}\label{lem4}
The following relations hold for $n, m\in \mathbb{N}$:
$$
t^n t_1^m \ \ \widehat{\simeq} \ \ A\cdot t^{n+m}\ +\ \underset{i=0}{\overset{n+m-1}{\sum}} A_i\ t^{i},
$$
\noindent where $A, A_i$ coefficients for all $i$.
\end{lemma}

\begin{proof}
We prove Lemma~\ref{lem4} by strong induction on the order of $t^n t_1^m \in \Lambda^{aug}$. The base of the induction is Lemma~\ref{lem1} for $n=1$. Assume that the relations are true for all elements in $\Lambda^{aug}$ of lower order than $t^n t_1^m$. Then, for $t^n t_1^m$ we have:

\[
t^n t_1^m  =  t^{n-1} (\underline{t t_1}) t_1^{m-1}\  = \ -\frac{1}{u^2 v^2} t^{n-1} (1+u\sigma_1 +vt+uv\left(\sigma_1 t + t \sigma_1) + u^2v \sigma_1 t \sigma_1 + uv^2 t \sigma_1 t \right) t_1^{m-1} \widehat{=}
\]
\[
\widehat{=}\ -\frac{1}{u^2 v^2} t^{n-1}t_1^{m-1}-\frac{1}{uv^2}t^{n-1}t_1^{m-1}\sigma_1 - \frac{1}{u^2v}t^{n}t_1^{m-1}-\frac{2}{uv}t^{n}t_1^{m-1}\sigma_1 -\frac{1}{v}t^{n-1}t_1^{m}-\frac{1}{u}t^{n+1}t_1^{m-1}\sigma_1.
\]

\smallbreak

The sum of the exponents in the elements $t^{n-1}t_1^{m-1}, t^{n}t_1^{m-1}$ and $t^{n-1}t_1^{m}$ on the right hand side of the relation are less than $n+m$, and thus, these elements are of lower order than $t^n t_1^m$ (recall Definition~\ref{order}). Applying now Theorem~\ref{tail} on the elements $t^{n-1}t_1^{m-1}\sigma_1, t^{n}t_1^{m-1}\sigma_1$ and $t^{n+1}t_1^{m-1}\sigma_1$, we convert them to sums of elements in $\Lambda^{aug}$ of lower order than $t^{n}t_1^{m}$. The proof is concluded by the induction hypothesis.
\end{proof}

\begin{thm}\label{main1}
Let $\tau\in \Lambda^{aug}_{(k)}$. The following relations hold:
\[
\tau\ \ \widehat{\simeq}\ \ \underset{i=0}{\overset{k}{\sum}} A_i\ t^i,
\]
\noindent where $A_i$ coefficients.
\end{thm}

\begin{proof}
Consider a monomial $\tau = t^{k_0}t_1^{k_1}\ldots t_n^{k_n} \in \Lambda^{aug}$. We prove the relations by strong induction on the order of $\tau$. The basis of the induction is Lemma~\ref{lem4}, since it deals with the monomials of type $t^nt_1^m$, which are of minimal order among all non-trivial monomials in $\Lambda^{aug}$. We assume that the statement of Theorem~\ref{main1} is true for all elements in $\Lambda^{aug}$ of lower order than $\tau$ and we will show that it is true for $\tau$. We have that:

\[
\begin{array}{lcl}
\tau & = & t^{k_0}t_1^{k_1}\ldots t_n^{k_n}\ =\ t^{k_0-1}\ (\underline{tt_1})\ t_1^{k_1-1}\ldots t_n^{k_n}\ =\\ 
&&\\
& = & t^{k_0-1}\left[-\frac{1}{(uv)^2}(1+u\sigma_1+vt+uv(\sigma_1t+t\sigma_1)+u^2v\sigma_1t\sigma_1+uv^2t\sigma_1t)\right]t_1^{k_1-1}\ldots t_n^{k_n}\ \widehat{=}\\
&&\\
& \widehat{=} & -\frac{1}{(uv)^2}\ t^{k_0-1}t_1^{k_1-1}\ldots t_n^{k_n}-\frac{1}{uv^2}\ t^{k_0-1}t_1^{k_1-1}\ldots t_n^{k_n}\sigma_1-\frac{1}{u^2v}\ t^{k_0}t_1^{k_1-1}\ldots t_n^{k_n}\ -\\
&&\\
& - & \frac{1}{uv}\ t^{k_0}t_1^{k_1-1}\ldots t_n^{k_n}\sigma_1-\frac{1}{v}\ t^{k_0-1}t_1^{k_1}\ldots t_n^{k_n}-
\frac{1}{u}\ t^{k_0+1}t_1^{k_1-1}\ldots t_n^{k_n}\sigma_1\ =\\
&&\\
& = & -\frac{1}{(uv)^2}\ t^{k_0-1}t_1^{k_1-1}\tau_{2, n}^{k_2, n}-\frac{1}{uv^2}\ t^{k_0-1}t_1^{k_1-1}\tau_{2, n}^{k_2, n}\sigma_1-\frac{1}{u^2v}\ t^{k_0}t_1^{k_1-1}\tau_{2, n}^{k_2, n}\ -\\
&&\\
& - & \frac{1}{uv}\ t^{k_0}t_1^{k_1-1}\tau_{2, n}^{k_2, n}\sigma_1-\frac{1}{v}\ t^{k_0-1}t_1^{k_1}\tau_{2, n}^{k_2, n}- \frac{1}{u}\ t^{k_0+1}t_1^{k_1-1}\tau_{2, n}^{k_2, n}\sigma_1
\end{array}
\]

\smallbreak

On the right-hand side of this relation we have the following monomials in $\Lambda^{aug}$:

$$t^{k_0-1}t_1^{k_1-1}\tau_{2, n}^{k_2, n}\ <\ t^{k_0}t_1^{k_1-1}\tau_{2, n}^{k_2, n}\ <\ t^{k_0-1}t_1^{k_1}\tau_{2, n}^{k_2, n}\ <\ \tau_{0, n}^{k_0, n}\ =\ \tau,$$

\noindent and the monomials $t^{k_0-1}t_1^{k_1-1}\tau_{2, n}^{k_2, n} \sigma_1$, $t^{k_0}t_1^{k_1-1}\tau_{2, n}^{k_2, n}\sigma_1$ and $t^{k_0+1}t_1^{k_1-1}\tau_{2, n}^{k_2, n} \sigma_1$ in the ${\rm H}_n(q)$-module $\Lambda^{aug}$. Applying Theorem~\ref{tail} on these monomials we have that:

\[
\begin{array}{lclll}
t^{k_0-1}t_1^{k_1-1}\tau_{2, n}^{k_2, n} \sigma_1 & \widehat{\simeq} & \underset{i}{\sum} A_i\ \tau_i, & {\rm such\ that}\ \ \tau_i\ <\ t^{k_0-1}t_1^{k_1-1}\tau_{2, n}^{k_2, n}\ <\  \tau_{0, n}^{k_0, n}, & {\rm for\ all}\ i    \\
&&&\\
t^{k_0}t_1^{k_1-1}\tau_{2, n}^{k_2, n} \sigma_1 & \widehat{\simeq} & \underset{j}{\sum} B_j\ \tau_j, & {\rm such\ that}\ \ \tau_j\ <\ t^{k_0}t_1^{k_1-1}\tau_{2, n}^{k_2, n}\ <\  \tau_{0, n}^{k_0, n}, & {\rm for\ all}\ j    \\
&&&\\
t^{k_0+1}t_1^{k_1-1}\tau_{2, n}^{k_2, n} \sigma_1 & \widehat{\simeq} & \underset{i}{\sum} C_i\ \tau_m, & {\rm such\ that}\ \ \tau_m\ <\ t^{k_0+1}t_1^{k_1-1}\tau_{2, n}^{k_2, n}\ <\  \tau_{0, n}^{k_0, n}, & {\rm for\ all}\ m 
\end{array}
\]

\noindent and thus, from the induction hypothesis the relation hold. 

\end{proof}

\subsection{Proof of Theorem~\ref{newbasis}}\label{pr}

Let $\tau^{\prime}_{0, n}\in \mathcal{B}^{\prime}({\rm ST}) \subset \Lambda_{(k)} \subset \Lambda^{aug}_{(k)}$. Then:

\[
\begin{array}{lcl}
\tau^{\prime}_{0, n} & \underset{Cor.~\ref{b'tol}}{\widehat{\simeq}} & \tau_{0, n}\ +\ \underset{i=0}{\sum} A_i\cdot \tau_i\ \underset{Prop.~\ref{ltob1}}{\widehat{\simeq}}\ A\cdot t^{index(\tau+1)}\ +\ \underset{i=0}{\overset{ind(\tau)}{\sum}} A_i\cdot t^i\ + \underset{i=0}{\sum} A_i\cdot \tau_i   \\
&&\\
&  \underset{Thm.~\ref{main1}}{\widehat{\simeq}}  & A\cdot t^{index(\tau+1)}\ +\ \underset{i=0}{\overset{ind(\tau)}{\sum}} A_i\cdot t^i\ + \ \underset{i=0}{\overset{k}{\sum}} B_i\cdot t^i \ = \ \underset{i}{\sum} C_i\cdot t^i\ \Rightarrow \\
&&\\
\tau^{\prime}_{0, n} & {\widehat{\simeq}} &  \underset{i=0}{\overset{n+1}{\sum}} C_i\cdot t^i,
\end{array}
\]

\smallbreak

\noindent where $A_i, B_i, C_i$ coefficients. Thus, elements in $\mathcal{B}^{\prime}({\rm ST})$ can be expressed as sums of elements in $\mathcal{B}({\rm ST})$, that is:

\bigbreak

{\it
The set $\mathcal{B}({\rm ST})$ spans the Kauffman bracket skein module of the solid torus.
}

\bigbreak

\noindent We now prove linear independence of the set $\mathcal{B}({\rm ST})$:

%\begin{thm}
%The set $\mathcal{B}({\rm ST})$ is linearly independent.
%\end{thm}
\bigbreak

The $t^n$'s geometrically consist of closed loops in the fundamental group of ST. Since $\pi_1({\rm ST}) = \mathbb{Z}$, $t^n \neq t^m$ for $n\neq m$ on the level of $\pi_1({\rm ST})$. This fact factors through to the Kauffman bracket skein module of ST, since the $t^n$'s can not be simplified neither by applying braid relations, nor by conjugation and stabilization moves. Moreover, the Tempereley-Lieb type crossing switches cannot be applied on the $t^n$'s, since they contain no classical crossings in our setting. Thus, the value of the KBSM(ST) on these elements remains the same as the value of the invariant $V^{{\rm B}}$ on these elements. 

\bigbreak

The proof of Theorem~\ref{newbasis} is now concluded.

\end{document}